\documentclass[11pt,leqno]{article}
\usepackage{amsmath, amscd, amsthm, amssymb, graphics, xypic, mathrsfs, setspace, fancyhdr, times}
\DeclareFontFamily{OT1}{pzc}{}
\DeclareFontShape{OT1}{pzc}{m}{it}%
             {<-> s * [1.195] pzcmi7t}{}
\DeclareMathAlphabet{\mathscr}{OT1}{pzc}%
                                 {m}{it}
\usepackage[pagebackref=true]{hyperref}
\hypersetup{backref}

\newcommand{\tensor}{\otimes}

\newcommand{\Spec}{\operatorname{Spec}}
\newcommand{\isomto}{{\stackrel{\sim}{\;\longrightarrow\;}}}
\newcommand{\isomt}{{\stackrel{{\scriptscriptstyle{\sim}}}{\;\rightarrow\;}}}

\renewcommand{\O}{{\mathcal O}}
\renewcommand{\hom}{\operatorname{Hom}}

\newcommand{\cplx}{{\mathbb C}}

\newcommand{\Z}{{\mathbb Z}}
\newcommand{\aone}{{\mathbb A}^1}

\newcommand{\et}{\text{\'et}}

\newcommand{\dmeff}{{\mathbf{DM}}^{\textrm{eff}}_{k,-}}

\newcommand{\Sm}{{\mathcal Sm}}

\newcommand{\K}{{{\mathbf K}}}

\newcommand{\Sym}{{\operatorname{Sym}}}

\newcounter{intro}
\theoremstyle{plain}
\newtheorem{thm}{Theorem}[section]

\newtheorem{lem}[thm]{Lemma}
\newtheorem{cor}[thm]{Corollary}
\newtheorem{prop}[thm]{Proposition}

\newtheorem{question}[thm]{Question}

\newtheorem*{thm*}{Theorem}
\newtheorem*{problem*}{Problem}

\newtheorem{thmintro}{Theorem}

\theoremstyle{definition}
\newtheorem{defn}[thm]{Definition}

\theoremstyle{remark}
\newtheorem{rem}[thm]{Remark}
\newtheorem{remintro}[thmintro]{Remark}

\numberwithin{equation}{section}

\begin{document}
\pagestyle{fancy}
\renewcommand{\sectionmark}[1]{\markright{\thesection\ #1}}
\fancyhead{}
\fancyhead[LO,R]{\bfseries\footnotesize\thepage}
\fancyhead[LE]{\bfseries\footnotesize\rightmark}
\fancyhead[RO]{\bfseries\footnotesize\rightmark}
\chead[]{}
\cfoot[]{}
\setlength{\headheight}{1cm}

\author{\begin{small}Aravind Asok\thanks{Aravind Asok was partially supported by National Science Foundation Award DMS-0900813.}\end{small} \\ \begin{footnotesize}Department of Mathematics\end{footnotesize} \\ \begin{footnotesize}University of Southern California\end{footnotesize} \\ \begin{footnotesize}Los Angeles, CA 90089-2532 \end{footnotesize} \\ \begin{footnotesize}\url{asok@usc.edu}\end{footnotesize}             }

\title{{\bf Rationality problems and \\ conjectures of Milnor and Bloch-Kato}}
\date{}
\maketitle

\begin{abstract}
We show how the techniques of Voevodsky's proof of the Milnor conjecture and the Voevodsky-Rost proof of its generalization the Bloch-Kato conjecture can be used to study counterexamples to the classical L\"uroth problem.  By generalizing a method due to Peyre, we produce for any prime number $\ell$ and any integer $n \geq 2$, a rationally connected, non-rational variety for which non-rationality is detected by a non-trivial degree $n$ unramified \'etale cohomology class with $\ell$-torsion coefficients.  When $\ell = 2$, the varieties that are constructed are furthermore unirational and non-rationality cannot be detected by a torsion unramified \'etale cohomology class of lower degree.
\end{abstract}


\section{Introduction}
Suppose $n$ is an integer coprime to the characteristic of $k$, and $L/k$ is a finitely generated separable field extension.  Consider the functor on $k$-algebras defined by $A \mapsto H^i_{\et}(A,\mu_n^{\tensor j})$ (we abuse terminology and write $A$ instead of $\Spec A$ for notational convenience).  Given a discrete valuation $\nu$ of $L/k$ with associated valuation ring $A$, one says that a class $\alpha \in H^i_{\et}(L,\mu_n^{\tensor j})$ is unramified at $\nu$ if $\alpha$ lies in the image of the restriction map $H^i_{\et}(A,\mu_n^{\tensor j}) \to H^i_{\et}(L,\mu_n^{\tensor j})$.  For any integers $i,j$, Colliot-Th\'el\`ene and Ojanguren \cite{CTO} defined the unramified cohomology group $H^i_{ur}(L/k,\mu_n^{\tensor j})$ as the subgroup of $H^i_{\et}(L,\mu_n^{\tensor j})$ consisting of those classes $\alpha$ that are unramified at every discrete valuation of $L$ trivial on $k$.  Colliot-Th\'el\`ene and Ojanguren also proved that the groups $H^i_{ur}(X/k,\mu_n^{\tensor j})$ are stable $k$-birational invariants \cite[Proposition 1.2]{CTO}.  If $X$ is any smooth proper $k$-variety, we write $H^i_{ur}(X/k,\mu_n^{\tensor j})$ for $H^i_{ur}(k(X)/k,\mu_n^{\tensor j})$.  For another point of view on these statements see \cite[Theorem 4.1.1]{CTPurity}.

The unramified cohomology groups can be used to detect counterexamples to the L\"uroth problem (over $\cplx$).  To set the stage, recall that in their celebrated work \cite{ArtinMumford}, building on a suggestion of C.P. Ramanujam, Artin and Mumford showed that the torsion subgroup of the singular cohomology group $H^3(X,\Z)$ is a birational invariant for smooth proper complex varieties $X$.  By explicitly exhibiting a conic bundle over a rational surface and a $2$-torsion class in this group, they produced an ``elementary" example of a unirational non-rational variety.

When $i = 2$, $j = 1$, and $L$ is the function field of a smooth proper $k$-variety $X$, the group $H^2_{ur}(L/k,\mu_n)$ admits a geometrically appealing interpretation.  In this case, one can identify the group $H^2_{ur}(L/k,\mu_n)$ with the $n$-torsion in the cohomological Brauer group of $X$ using results of Auslander-Goldman and Grothendieck ({\em cf}. \cite[\S6 and \S7]{GrothBrauerIII}).  In this context, it was Saltman who first applied the unramified Brauer group to rationality problems  \cite{Saltman1,Saltman2}, and Saltman mentions a relationship between the Artin-Mumford example and the unramified Brauer group.  In fact, as Colliot-Th\'el\`ene and Ojanguren \cite[\S 2]{CTO} make explicit, the $2$-torsion class constructed by Artin and Mumford is a $2$-torsion class in the cohomological Brauer group (see \cite[\S 8]{Ojanguren} for a nice overview of this circle of ideas).

Colliot-Th\'el\`ene and Ojanguren went further and produced an example of a unirational non-rational variety where non-rationality was detected by means of a degree $3$ unramified cohomology class in the function field of a quadric over a transcendence degree $3$ extension of $\cplx$.  Furthermore, they proved that the cohomological Brauer group of their example was trivial.  They also suggested that it should be possible to produce higher dimensional varieties that were unirational, not-stably rational, where non-rationality is detected by a class of degree $n$, yet could not be detected by lower degree classes.

A key algebraic input in the construction of \cite{CTO} was a result of Arason \cite[p. 489]{Arason} describing the kernel of the restriction map $H^3_{\et}(k,\mu_2^{\tensor 2}) \to H^3_{\et}(L/k,\mu_2^{\tensor 2})$ for $L$ the function field of a special quadratic form: a (neighbor of a) Pfister $3$-form.  Using this result, the degree $3$ unramified invariant was constructed by an intricate geometric argument. Later, using a generalization of Arason's theorem due to Jacob and Rost \cite[p. 559]{JacobRost}, Peyre \cite{Peyre1} produced examples of fields where non-rationality was detected by a degree $4$ unramified cohomology class.  While he showed that the Brauer group was trivial for his examples, he did not establish vanishing of degree $3$ unramified cohomology.  The main result of this paper shows that the construction of Colliot-Th\'el\`ene-Ojanguren, as generalized by Peyre, extends to all degrees.

\begin{thmintro}[See Theorem \ref{thm:nonvanishing}]
\label{thmintro:mainexample}
For every integer $n > 0$, there exists a smooth projective complex variety $X$ that is unirational, for which $H^i_{ur}(X,\mu_2^{\tensor i})$ vanishes for each $i < n$, yet $H^n_{ur}(X,\mu_2^{\tensor n}) \neq 0$, and so $X$ is not $\aone$-connected (nor stably $k$-rational).
\end{thmintro}

The fields used to construct the examples are function fields of products of quadrics defined by Pfister $n$-forms over rational function fields in many variables.  The result of Arason used by Colliot-Th\'el\`ene-Ojanguren (and the subsequent generalization by Jacob-Rost) was extended to all (neighbors of) Pfister $n$-forms by Orlov-Vishik-Voevodsky \cite[Theorem 2.1]{OrViVo} as a consequence of Voevodsky's spectacular affirmation of Milnor's conjecture that the mod $2$ norm residue homomorphism is an isomorphism \cite{VMilnor}; these results provide the main new ingredient in construction of our examples.

\begin{remintro}
As was pointed out to the author by the referee of a previous version of this paper, examples of the sort produced here can also be produced by combining results of Izhboldin \cite{Izhboldin} with \cite{OrViVo}; see \cite[Lemma 9.5]{Izhboldin}, which makes reference to \cite[Lemma 6.4 and Lemma 8.12]{Izhboldin}.  Nevertheless, our approach, following Peyre, is different.  Moreover, examples such as the above have not (to the best of the author's knowledge) appeared explicitly anywhere in print.
\end{remintro}

Peyre also observed that for $\ell$ a prime different from $2$, it was possible to construct unirational non-rational varieties whose non-rationality was detected by degree $n$ unramified cohomology classes when $n$ was small.  Indeed, he observed \cite[Proposition 5]{Peyre1} that the Brauer-Severi varieties attached to cyclic algebras can be used to construct unirational varieties with a non-trivial degree $2$ unramified $\ell$-torsion class.  He also observed \cite[Proposition 5]{Peyre1} that it is possible to construct unirational varieties where non-rationality is detected by a degree $3$ unramified $\ell$-torsion class, yet for which the (unramified) Brauer group is trivial.  Again, each example relied on a variant of the result of Arason mentioned above (in the first case, the necessary result was provided by Amitsur and in the second case by Suslin; see \cite{Peyre1} for precise references).  As a partial generalization of these results, we establish the following.

\begin{thmintro}[See Theorem \ref{thm:rcrostvarieties}]
\label{thmintro:lgeneric}
For any prime number $\ell$ and any integer $n \geq 2$, there exists a smooth projective rationally connected complex variety $X$ such that $H^n_{ur}(X/\cplx,\mu_{\ell}^{\tensor n})$ is non-trivial.  In particular, $X$ is not $\aone$-connected (nor stably $k$-rational).
\end{thmintro}

If $\ell$ is a prime number different from $2$, in the course of the proof of the Bloch-Kato conjecture certain new varieties---$\ell$-generic splitting varieties---were introduced to assume the role played by Pfister quadrics in the proof of the Milnor conjecture.  The varieties used by Peyre in his constructions are precisely of this form.  Theorem \ref{thmintro:lgeneric} relies on the explicit constructions of $\ell$-generic splitting varieties due to Voevodsky and Rost \cite{Rost, SuslinJoukhovitski}, and results of \cite{MerkurjevSuslin} that provide partial analogs for the results of Orlov-Vishik-Voevodsky, Amitsur or Suslin mentioned above.  We prove rational connectivity of some explicit $\ell$-generic splitting varieties and again apply the method of Peyre to construct non-trivial unramified classes.

The paper closes with a discussion of the unramified cohomology of smooth proper rationally connected varieties.  After recalling some vanishing statements, we ask some questions about when unramified cohomology can be non-trivial.  Even for smooth proper rationally connected $3$-folds, the answer to these questions seems to be unknown in general.

\subsubsection*{Acknowledgements}
This paper was originally a piece of \cite{ABirational}, and the current version is a revision and extension of results that originally were stated there.  We thank the referees of the old version for many helpful comments and corrections.  Our interest in these examples arose from a desire to understand the zeroth $\aone$-homology sheaf of a smooth proper variety, which was an outgrowth of joint work with Fabien Morel \cite{AM}; we thank him for much encouragement and many inspiring discussions.  We thank Sasha Merkurjev for pointing out a simplification of our original proof of rational connectivity of Rost varieties, and Jean-Louis Colliot-Th\'el\`ene for extremely helpful correspondence, especially regarding the questions at the end of the paper.  Finally, we thank Christian Haesemeyer for useful conversations and Zinovy Reichstein for encouragement.

\section{Unramified cohomology and splitting varieties}
Throughout this work, the letters $F,k,$ or $K$ will be used to denote fields containing $\cplx$ as subfields; while this restriction is not necessary everywhere, it simplifies the presentation in a number of places.  In particular, if $\ell$ is a prime number, we will fix, once and for all, a primitive $\ell$-th root of unity $\tau$ and consequently an isomorphism $\Z/\ell \cong \mu_{\ell}$.  Throughout the paper $K^M_*(F)$ denotes the graded Milnor K-theory ring of the field $F$ and $K^M_*(F)/n$ denotes the mod $n$ Milnor K-theory ring of $F$.

We write $\Sm_k$ for the category of schemes separated, smooth, and finite type over $k$; we will also refer to elements of $\Sm_k$ as varieties.  We view $\Sm_k$ as a site by equipping it with the Nisnevich topology; the word sheaf will, unless otherwise mentioned, mean Nisnevich sheaf of on $\Sm_k$.

\subsubsection*{The norm-residue isomorphism theorem and consequences}
Kummer theory gives a canonical identification $k^*/(k^*)^n \isomt H^1_{\et}(\Spec k,\mu_n)$.  Given a sequence of elements $a_1,\ldots,a_m \in k^*$, we write $(a_i)$ for the image of $a_i \in H^1_{\et}(\Spec k,\mu_n)$, and $(a_1,\ldots,a_m)$ or $\underline{a}$ (depending on context) for the cup product $(a_1) \cup \cdots \cup (a_m) \in H^m_{\et}(\Spec k,\mu_n^{\tensor m})$.

By definition, there is an induced identification $K^M_1(k)/n \isomto k^*/(k^*)^n$.  Combining this identification with the canonical isomorphism given by Kummer theory in the previous paragraph, there is a canonical isomorphism $K^M_1(k)/n \isomt H^1_{\et}(\Spec k,\mu_n)$.  Bass and Tate showed that the previous identification extends to a morphism of graded rings
\[
\psi_{n,*}: K^M_*/n(k) \isomto H^*_{\et}(\Spec k,\mu_n^{\tensor *}),
\]
and $\psi_{n,*}$ or its graded components are referred to as the norm-residue homomorphism.

The norm residue homomorphism is functorial in $k$, but also satisfies additional compatibility conditions.  If $k$ is a field equipped with a discrete valuation $\nu$ and associated residue field $\kappa_{\nu}$, Milnor constructed in \cite[Lemma 2.1]{Milnor} residue maps
\[
\partial_{\nu}: K^M_i/n(k) \longrightarrow K^M_{i-1}/n(\kappa_{\nu}).
\]
Using cohomological purity, there are corresponding residue maps in \'etale cohomology and the norm residue homomorphism is compatible with residue maps.  Furthermore, both Milnor K-theory and \'etale cohomology admit transfers, and the norm residue homomorphism is compatible with transfers as well.  These compatibilities lead to a sheafification of the norm residue homomorphism; this circle of ideas is closely connected with the results of Bloch-Ogus (see, e.g., \cite{CTHK} for discussion of this point of view).

For a smooth irreducible $k$-scheme $V$ with function field $k(V)$, one defines the unramified Milnor K-group $\K^{M}_i/n(V)$ as the intersection of kernels of residue maps coming from the geometric discrete valuations of $k(V)$.  One defines $\K^M_i/n(U)$ on a not necessarily irreducible smooth $k$-scheme as a sum.  The functor $U \mapsto \K^M_i/n(U)$ then defines a Nisnevich sheaf $\K^M_i/n$ on $\Sm_k$ that is often called the {\em $i$-th unramified mod $n$ Milnor K-theory sheaf}.  Furthermore, $\K^M_i/n$ is actually a homotopy invariant Nisnevich sheaf with transfers and therefore is strictly $\aone$-invariant (recall that a Nisnevich sheaf of abelian groups ${\mathbf A}$ is strictly $\aone$-invariant if its cohomology presheaves $V \mapsto H^i_{Nis}(V,{\mathbf A})$ are $\aone$-invariant for any smooth $k$-scheme $V$; this notion is studied in great detail in \cite[\S 6]{MStable}).

Similarly, the Nisnevich sheafification of the functor on smooth $k$-schemes defined by $U \mapsto H^i_{\et}(U,\mu_n^{\tensor j})$ will be denoted ${\mathcal H}^i_{\et}(\mu_n^{\tensor j})$, and the sections of this sheaf over a smooth scheme $X$ can be identified by a remark of Gabber with $H^i_{ur}(X,\mu_n^{\tensor j})$; this sheaf is called the {\em $i$-th mod $n$ weight $j$ unramified \'etale cohomology sheaf}.  Again, the sheaf ${\mathcal H}^i_{\et}(\mu_n^{\tensor j})$ is strictly $\aone$-invariant.

Compatibility of the norm residue isomorphism with residues gives a morphism of sheaves
\begin{equation}
\label{eqn:normresidue}
\psi_{n,*}: \K^M_i/n \longrightarrow {\mathcal H}^i_{\et}(\mu_n^{\tensor i}).
\end{equation}
Since strictly $\aone$-invariant sheaves have Gersten resolutions (see, e.g., \cite{CTHK}), and one deduces immediately from this fact that a morphism of strictly $\aone$-invariant sheaves (with transfers) is an isomorphism if and only if it is an isomorphism on sections over finitely generated separable extensions $L/k$.  Therefore, the morphism $\psi_{n,*}$ is an isomorphism if and only if the norm residue isomorphism conjecture holds for all extensions $L/k$.

Voevodsky's proof of the Milnor conjecture and the subsequent Voevodsky-Rost proof of the Bloch-Kato conjecture imply the following result.

\begin{thm}[{\cite[Corollary 7.5]{VMilnor}, \cite[Theorem 6.1]{VModl}}]
\label{thm:normresidue}
The morphism $\psi_{n,*}$ from \textup{\ref{eqn:normresidue}} is an isomorphism of sheaves.
\end{thm}

One immediate consequence of this form of the theorem is that unramified mod $n$ Milnor K-theory can be canonically identified with unramified \'etale cohomology.  As usual, write $\Z/\ell(n)$ denote the mod $n$ motivic complex (see, e.g., \cite{MVW}).  Our choice of an $\ell$-th root of unity determines a morphism of motivic complexes $\tau: \Z/\ell(n-1) \to \Z/\ell(n)$.  Suslin and Voevodsky showed that the Beilinson-Lichtenbaum conjecture was equivalent to the Bloch-Kato conjecture \cite{SuslinVoevodskyBK}.  As a consequence, one deduces the existence of the following triangle in Voevodsky's derived category of motives.

\begin{lem}[{\cite[Lemma 2.4]{OrViVo}}]
\label{lem:distinguished}
The morphism $\tau$ extends to a distinguished triangle in $\dmeff$ of the form
\[
\Z/\ell(n-1) \stackrel{\tau}{\longrightarrow} \Z/\ell(n) \longrightarrow \underline{H}^n(\Z/\ell(n))[-n].
\]
\end{lem}

The sheaf $\underline{H}^n(\Z/\ell(n))$ is strictly $\aone$-invariant and is therefore determined by its sections over fields.  The computation of the $(n,n)$-motivic cohomology of fields (see, e.g., \cite[\S 4]{MVW}) together with Theorem \ref{thm:normresidue} allows us to identify $\underline{H}^n(\Z/\ell(n)) \isomt {\mathcal H}^n_{\et}(\mu_{\ell}^{\tensor n})$.

If $X$ is any smooth $k$-scheme, the \u Cech object associated with $X$, denoted $\breve{C}(X)$ is the simplicial scheme whose $n$-th term is the $n+1$-fold fiber product of $X$ with itself over $\Spec k$; the relevant simplicial structures are induced by diagonal maps and partial projections.  The canonical morphism $\breve{C}(X) \to \Spec k$ is an isomorphism in the (Nisnevich) simplicial homotopy category if it is an epimorphism of Nisnevich sheaves \cite[\S 2 Lemma 1.15]{MV}, i.e., if $X$ has a $k$-rational point.  Whether or not $X$ has a $k$-rational point, the map $\breve{C}(X) \to \Spec k$ induces an isomorphism $H^{n,n}(\Spec k,\Z/\ell) \isomt H^{n,n}(\breve{C}(X),\Z/\ell)$ (see, e.g., \cite[Lemma 2.2]{OrViVo}).

The diagonal map $X \to \breve{C}(X)$ induces a morphism
\[
H^0(\breve{C}(X),{\mathcal H}^n_{\et}(\mu_{\ell}^{\tensor n})) \longrightarrow H^0(X,{\mathcal H}^n_{\et}(\mu_{\ell}^{\tensor n})) = H^n_{ur}(X/k,\mu_{\ell}^{\tensor n}).
\]
Applying the functor ${\bf M}(\breve{C}(X),-)$ to the triangle in Lemma \ref{lem:distinguished}, one can deduce the following result.

\begin{lem}[{\cite[Proof of Lemma 6.5]{VModl}}]
\label{lem:kernelofunramified}
There is an exact sequence of the form
\[
H^{n,n-1}(\breve{C}(X),\Z/\ell) \longrightarrow H^{n}_{\et}(\Spec k,\Z/\ell) \longrightarrow H^n_{ur}(X/k,\mu_{\ell}^{\tensor n}).
\]
\end{lem}

\subsubsection*{Small Pfister quadrics and the results of Orlov-Vishik-Voevodsky}
We use standard notation from the theory of quadratic forms.  For elements $a_1,\ldots,a_m \in k^*$ write $\langle a_1 \rangle$ for the $1$-dimensional quadratic form $a_1 x^2$, $\langle a_1,\ldots,a_m \rangle$ for the orthogonal sum $\langle a_1 \rangle \oplus \cdots \oplus \langle a_m \rangle$, $\langle \langle a_1 \rangle \rangle$ for the form $\langle 1 , -a \rangle$ and $\langle \langle a_1,\ldots,a_m \rangle \rangle$ for the tensor product $\langle \langle a_1 \rangle \rangle \tensor \cdots \tensor \langle \langle a_m \rangle \rangle$.

\begin{defn}
\label{defn:pfisterforms}
Given an $n$-tuple $a_1,\ldots,a_n \in k^*$, the {\em small Pfister quadric} attached to the symbol $\underline{a}$ is the smooth projective variety $Q_{\underline{a}}$ defined by the homogeneous equation
\[
\langle \langle a_1,\ldots,a_{n-1} \rangle \rangle = \langle a_n \rangle.
\]
We write $\mathcal{X}_{\underline{a}}$ for $\breve C(Q_{\underline{a}})$.
\end{defn}

Orlov-Vishik-Voevodsky give a strengthening of Lemma \ref{lem:kernelofunramified}, which we summarize in the following statement (this form of the result, which was perhaps contained in a preprint version of \cite{OrViVo}, has apparently been relegated to the status of mathematical folklore).

\begin{thm}[{\cite[Theorem 2.1]{OrViVo}}]
\label{thm:ovv}
If $\underline{a}$ is a non-trivial symbol of length $n$, the restriction map
\[
\eta^i: H^{i}_{\et}(\Spec k,\Z/2) {\longrightarrow} H^i_{ur}(Q_{\underline{a}}/k,\mu_2^{\tensor i})
\]
is an isomorphism for $i < n$ and has kernel $\Z/2$ generated by the class of $\underline{a}$ if $i = n$.
\end{thm}

\begin{proof}[Remarks on the proof.]
Any quadric that has a $k$-rational point is in fact $k$-rational.  The quadrics $Q_{\underline{a}}$ are therefore {\em generically rational} in the sense that they are rational over their generic point.  From Lemma \ref{lem:distinguished}, and the argument of \cite[Theorem A.1]{Kahngenericallycellular} we get a description of both the kernel and cokernel of the restriction map in unramified cohomology.

The description of the kernel (and triviality when $i < n$) is an immediate consequence of \cite[Theorem 2.1]{OrViVo}.  The vanishing of the cokernel is proven by means of the same techniques.  See, e.g., \cite[\S 3 Proof of Theorem 3]{KahnSujatha} for a detailed explanation of the triviality of the cokernel.
\end{proof}

\subsubsection*{Constructing Rost varieties}
We now discuss partial analogs of these constructions adapted to primes other than $2$.  We start by outlining a construction given in \cite[\S 2]{SuslinJoukhovitski}, which is due originally to Voevodsky.  Start with any smooth projective $k$-variety $Y$.  Write $Y^m$ for the $m$-fold fiber product of $Y$ with itself over $m$.  Let $\Sym^m(Y)$ be the $m$-th symmetric product of $Y$, i.e., the (singular in general) quotient $Y^m/S_m$, where the symmetric group $S_m$ acts by permuting the factors. The symmetric group $S_m$ acts freely on the open subscheme $Y^m \setminus \Delta$ where $\Delta$ is the union of the diagonals, which is the locus corresponding to $m$ distinct points in $Y$.  A geometric quotient $C_m(Y) := (Y^m \setminus \Delta)/S_m$ exists as a smooth scheme, and the notation is chosen to indicate that $C_m(Y)$ is the configuration space of $m$ points on $Y$.

There is an addition map $p: Y \times \Sym^{m-1}Y \to \Sym^m(Y)$; this morphism is finite surjective of degree $m$.  The preimage of $C_{m}(Y)$ under $p$ gives rise to a morphism
\[
p^{-1}(C_{m}(Y)) \subset Y \times S^{m-1}(Y) \longrightarrow C_m(Y)
\]
that is a finite \'etale morphism of degree $m$.  In particular, $p^{-1}(C_{m}(Y))$ is itself a smooth scheme.

The sheaf ${\mathcal A} := p_*(\O_{Y \times \Sym^{m-1}(Y)})|_{C_{m}(Y)}$ is a locally free sheaf of $\O_{C_{m}(Y)}$-algebras of rank $m$.  We let ${\mathbb V}({\mathcal A}) \to C_{m}(Y)$ be the associated geometric vector bundle (given by taking the spectrum of the symmetric algebra of the dual of ${\mathcal A}$).  Since ${\mathcal A}$ is a locally free sheaf of algebras, there is a well-defined norm morphism
\[
N: \mathcal{A} \longrightarrow \O_{C_{m}(Y)},
\]
which can be identified as a section of $\Sym^m{\mathcal A}^{\vee}$.

Now, take an element $a \in k^*$ and view this as a section of $\O_{C_{m}(Y)}$.  Then, consider the closed subscheme of $W \subset {\mathcal A}$ defined by the equation $N - a = 0$.  By \cite[Lemma 2.1]{SuslinJoukhovitski}, the variety $W$ is smooth over $C_{m}(Y)$ and therefore smooth over $k$ (and also geometrically irreducible).  Let $N(Y,a,m)$ be a non-singular compactification of $W$, which exists since $k$ has characteristic $0$; by resolving indeterminacy, we can even assume $N(Y,a,m)$ is proper and surjective over a non-singular compactification of $C_m(Y)$.  Note that $\dim_k N(Y,a,m) = m\dim Y+ m - 1$.

\begin{prop}
\label{prop:rostvarietiesrcprop}
If $k$ is any algebraically closed field containing $\cplx$, and $Y$ is smooth projective unirational (resp. rationally connected) variety, then $N(Y,a,m)$ is also smooth projective unirational (resp. rationally connected) variety.
\end{prop}

\begin{proof}
If $Y$ is any unirational (resp. rationally connected) smooth proper variety, then $Y^m$ is again unirational (resp. rationally connected).  Let $E_{a}$ be the subvariety of ${\mathbb A}^m$ defined by the equation $\prod_i z_i = a$, which is itself rational; this equation is precisely $N = a$ for a split \'etale algebra of dimension $m$.  Note that there is an obvious action of the symmetric group $S_m$ on $E_a$ by permuting the coordinates in ${\mathbb A}^m$.

The morphism $Y^m \setminus \Delta \to C_m(Y)$ is an $S_m$-torsor: this morphism sends a sequence of distinct points in $Y^m \setminus \Delta$ to the corresponding $0$-cycle.  The pullback of the \'etale $\O_{C_m(Y)}$-algebra ${\mathcal A}$ along the morphism $Y^m \setminus \Delta \to C_m(Y)$, is a split \'etale $\O_{Y^m \setminus \Delta}$-algebra by construction.  Consequently, the pullback of $W$ along the morphism $Y^m \setminus \Delta \to C_m(Y)$ is precisely $Y^m \setminus \Delta \times E_a$.  Moreover, the group $S_m$ acts diagonally on the product $Y^m \setminus \Delta \times E_a$, and the quotient is precisely $W$.  In particular, the quotient map gives a dominant morphism $Y^m \setminus \Delta \times E_a \to W$ that has degree $m$ (see \cite[5c-d]{KarpenkoMerkurjev} for a more general version of this construction).

Since $E_a$ is rational, any smooth proper compactification $\overline{E_a}$ is again rational.  By resolving indeterminacy of the rational map $Y^m \times \overline{E}_a \to N(Y,a,m)$, we obtain a dominant morphism from a unirational (resp. rationally connected) variety to $N(Y,a,m)$, which means that $N(Y,a,m)$ is unirational (resp. rationally connected) as well.
\end{proof}

\begin{cor}
\label{cor:rostvarietiesarerationallyconnected}
Suppose $k$ is the function field of a smooth complex rationally connected variety, $a \in k^*$, $Y$ is a smooth projective $k$-variety, and $K = k(N(Y,a,m))$.  If $X$ is any smooth proper model of $K/\cplx$, then $X$ is rationally connected.
\end{cor}

\begin{proof}
Suppose $Z$ is any smooth projective variety with $\cplx(Z) = k$.  By assumption $Z$ is rationally connected.  By construction, there is a dominant rational map from $X$ to $Z$, so by resolving indeterminacy, there exist a smooth proper variety $X'$, a morphism $X' \to X$ that is a birational equivalence, and a projective morphism $X' \to Z$.  Since $Z$ is rationally connected it suffices by \cite[Corollary 1.3]{GHS} to establish that the general fibers of $X' \to Z$ are rationally connected.

By clearing the denominators, the $k$-variety $N(Y,a,m)$ gives rise to a variety $Y'$ and a smooth and proper morphism $Y' \to U$ for an open subscheme $U \subset Z$.  Tracing through the constructions, one observes that the fiber over any closed point $u \in U$ of $Y' \to U$ is a smooth compactification of $N(Y_u,a_u,m)$ where $a_u$ is the specialization of $a$ to $u$.  Thus, by Proposition \ref{prop:rostvarietiesrcprop}, over any closed point $u \in U$ the fiber of $Y' \to U$ is rationally connected, so by \cite[Theorem IV.3.11]{Kollar} every fiber of this morphism is rationally connected.  It follows immediately that the general fiber of $X' \to Z$ is rationally connected as well.  Therefore, $X$ is rationally connected.
\end{proof}

\subsubsection*{Rost varieties and their motivic cohomology}
Now, suppose $\ell$ is a prime number.  Given elements $a_1,\ldots,a_m \in k^*$, we can also consider the symbol $\underline{a} \in K^M_m(k)/\ell$.  An extension $L/k$ is called a {\em splitting field} for $\underline{a}$ if $\underline{a}$ becomes trivial in $K^M_m(L)/\ell$.  An irreducible smooth $k$-variety $X$ is a {\em splitting variety} for a symbol $\underline{a}$ modulo $\ell$ if $k(X)$ is a splitting field for $\underline{a}$.  An irreducible smooth $k$-variety $X$ is called a {\em generic splitting variety} for $\underline{a}$ modulo $\ell$ if $X$ is a splitting variety, and for any splitting field $L$ of $\underline{a}$, $X(L)$ is non-empty.  Unfortunately, generic splitting varieties are only known to exist for $\ell = 2$ (Pfister quadrics) or $n \leq 3$ (when $n = 2$, the Brauer-Severi varieties attached to cyclic algebras provide models, and when $n = 3$, the Merkurjev-Suslin varieties provide models).

An irreducible smooth $k$-variety $X$ is an {\em $\ell$-generic splitting variety} for a non-zero symbol $\underline{a} \in \K^M_m(k)/\ell$ if $X$ is a splitting variety for $\underline{a}$, and for any splitting field $L$ of $\underline{a}$, there is a finite extension $L'/L$ of degree prime to $\ell$ such that $X(L')$ is non-empty.  An $\ell$-generic splitting variety for a non-zero symbol $\underline{a} \in \K^M_m(k)/\ell$ is a {\em Rost variety} for the symbol $\underline{a}$ if it satisfies the conditions of \cite[Definition 0.5]{HaesemeyerWeibel}

For $\ell \neq 2$, Rost varieties for a symbol of length $m$ can be constructed inductively assuming the Bloch-Kato conjecture in weights $< m$.  For the symbol $\{a_1,a_2\}$, the Severi-Brauer variety associated with the cyclic algebra attached to $\{a_1,a_2\}$ is even a generic splitting variety.  Given an arbitrary symbol $\underline{a} = \{a_1,\ldots,a_m\}$, one defines inductively $Y_{\underline{a}} = Y_{(a_1,\ldots,a_n)} := N(Y_{a_1,\ldots,a_{n-1}},a_n,l)$.  Note that $\dim_k Y_{\underline{a}} = \ell^{n-1} - 1$.

\begin{thm}[Voevodsky, Rost ({\cite[Theorem 1.21]{SuslinJoukhovitski}})]
\label{thm:rostvarietiesexist}
If $\ell$ is a prime number, and $\underline{a} \in K^M_n(k)/\ell$ is a non-trivial symbol, the variety $Y_{\underline{a}}$ is a Rost variety for the symbol $\underline{a}$.
\end{thm}

\begin{rem}
\label{rem:retractrational}
Repeatedly applying Proposition \ref{prop:rostvarietiesrcprop}, one sees that the varieties $Y_{\underline{a}}$ are, after passing to an algebraic closure of the base field, rationally connected.  However, it is to the best of the author's knowledge not known whether the varieties $Y_{\underline{a}}$ inherit any finer rationality properties from the construction.  For example, it is not known in general whether $Y_{\underline{a}}$ is a retract $k$-rational variety for an arbitrary prime $\ell$ and symbol in $\underline{a} \in K^M_n(k)/\ell$.
\end{rem}

By \cite[Lemma 6.4]{VModl}, if the symbol $\underline{a}$ is non-trivial in $K^M_n(k)/\ell$, then its image in $H^n_{\et}(k,\mu_{\ell}^{\tensor n})$ under $\psi_{n,\ell}$ remains non-zero.  If $Y_{\underline{a}}$ is a Rost variety attached to a symbol $\underline{a}$, it follows from this observation that the class of $\underline{a}$ defines a non-trivial element in $H^{n,n-1}(\breve{C}(Y_{\underline{a}},\Z/\ell)$, which by Lemma \ref{lem:kernelofunramified} is precisely equal to the kernel of the map $H^{n}_{\et}(\Spec k,\Z/\ell) \to H^n_{ur}(Y_{\underline{a}}/k,\mu_{\ell}^{\tensor n}))$.  Merkurjev and Suslin computed the motivic cohomology of $\breve{C}(Y_{\underline{a}})$.  We summarize one consequence of their main result, which provides a partial analog of Theorem \ref{thm:ovv} in the case where $\ell \neq 2$.

\begin{thm}[{\cite[Proposition 1.4 and Theorem 1.15]{MerkurjevSuslin}}]
\label{thm:merkurjevsuslin}
If $\underline{a}$ is a non-trivial symbol, and $Y_{\underline{a}}$ is a corresponding Rost variety, the kernel of the map $H^{n}_{\et}(\Spec k,\Z/\ell) \longrightarrow H^n_{ur}(Y_{\underline{a}}/k,\mu_{\ell}^{\tensor n})$ is $\Z/\ell$ generated by the class of the symbol $\underline{a}$.
\end{thm}




\section{Vanishing results}
The goal of this section is to prove some vanishing results for unramified cohomology.

\begin{thm}
\label{thm:vanishing}
Fix an integer $n > 1$.  Let $k$ be the function field of any variety $Y$ over $\cplx$ for any integer $m$) such that $H^i_{ur}(Y/\cplx,\mu_2^{\tensor i}) = 0$ for $i < n$.  If $\underline{a}$ is a symbol of length $n$, the group $H^i_{ur}(k(Q_{\underline{a}})/\cplx,\mu_2^{\tensor i})$ is trivial for $i \leq n-1$.
\end{thm}

The proof of Theorem \ref{thm:vanishing} will require us to understand residue maps attached to discrete valuations on $k(Q_{\underline{a}})$.  To this end, we will use the following lemma, which is a direct consequence of \cite[Corollary 6.13]{EKM}.

\begin{lem}
\label{lem:rewriting}
Suppose $K$ is a field (assumed again to have characteristic $0$), and $\nu$ is a discrete valuation of $K$ with associated local ring $\O_{\nu}$ and residue field $\kappa_{\nu}$.  If $a_1,\ldots,a_n$ are elements of $K^*$, there are elements $b_1,\ldots,b_n$ such that $b_1,\ldots,b_{n-1}$ are elements of $\O_{\nu}^*$ and such that over $K$, the quadratic forms $\langle \langle a_1,\ldots, a_n \rangle \rangle$ and $\langle \langle b_1,\ldots,b_n \rangle \rangle$ are $K$-isomorphic.
\end{lem}

\begin{proof}[Proof of Theorem \ref{thm:vanishing}]
Let $K = k(Q_{\underline{a}})$ and suppose $X$ is a smooth proper model of $X$ over $\cplx$.  The group $H^i_{ur}(X/\cplx,\mu_2^{\tensor i})$ is a subgroup of $H^i_{\et}(K/k,\mu_2^{\tensor i})$ by its very definition.  By Theorem \ref{thm:ovv}, the inclusion map
\[
H^i_{\et}(k,\mu_2^{\tensor i}) \longrightarrow H^i_{ur}(Q_{\underline{a}}/k,\mu_2^{\tensor i})
\]
is injective for $i \leq n-1$.

Suppose $\nu$ is a geometric discrete valuation of $k$ with associated valuation ring $\O_{\nu}$, residue field $\kappa_{\nu}$, and let $\pi$ be a local parameter.  Applying Lemma \ref{lem:rewriting}, after multiplying by an element of $k$ if necessary, we can assume that the symbol $\underline{a}$ giving $Q_{\underline{a}}$ has $a_1,\ldots,a_{n-2} \in \O_{\nu}^*$.  The resulting variety can be viewed as a subscheme of ${\mathbb P}^N_{\O_{\nu}}$, and reducing the resulting equation modulo the maximal ideal of $\O_{\nu}$ gives rise to a geometric discrete valuation $\nu'$ of $K$ with residue field $\kappa_{\nu'}$.  Moreover, the local homomorphism $\O_{\nu} \to \O_{\nu'}$ is necessarily unramified, and there is a commutative diagram of the form
\[
\xymatrix{
H^i_{\et}(k,\mu_2^{\tensor i}) \ar[r]\ar[d] & H^{i-1}_{\et}(\kappa_{\nu},\mu_2^{\tensor i-1}) \ar[d] \\
H^i_{\et}(K,\mu_2^{\tensor i}) \ar[r] & H^{i-1}_{\et}(\kappa_{\nu'},\mu_2^{\tensor i-1}).
}
\]
If $i < n$, the left hand morphism is an injective by Theorem \ref{thm:ovv}.  We claim the right hand morphism is injective as well.  Indeed, the chosen form of the equation shows that $\kappa_{\nu'}$ is a function field of either a small Pfister $n$-form over $\kappa_\nu$ or a purely transcendental extension of the function field a small Pfister $n-1$-form over $\kappa_{\nu}$.  Since $i-1 < n-1 < n$, in the first case, injectivity of the right hand arrow follows from another application of Theorem \ref{thm:ovv}.  In the second case, injectivity for a purely transcendental extension is clear, and injectivity of the right hand arrow then follows by yet another application of Theorem \ref{thm:ovv}.

Suppose we take a class in $H^i_{ur}(X/\cplx,\mu_2^{\tensor i})$ and view it as an element of $H^i_{\et}(k,\mu_2^{\tensor i})$ by means of the isomorphism of Theorem \ref{thm:ovv}.  A diagram chase show Such a class is unramified at every discrete valuation $\nu'$ of $K$ of the form in the previous paragraph and therefore comes from a class in $H^i_{ur}(k/\cplx,\mu_i^{\tensor 2})$.  Since $k$ is the function field of an $\aone$-connected variety, it follows that the group $H^i_{ur}(k/\cplx,\mu_i^{\tensor 2})$ is trivial and therefore also that the original class is trivial as well.
\end{proof}

\begin{rem}
The hypotheses of the theorem are satisfied for $Y$ any $\aone$-connected smooth complex variety by \cite[Proposition 3.5 and Lemma 4.7]{ABirational}.  In what follows, we will only use the case where $Y$ is projective space, but as shown in \cite[Theorem 2.3.6.ii]{AM}, any stably rational (or, more generally, retract rational) smooth proper complex variety is $\aone$-connected.  Inductively applying the theorem, one deduces the following result.
\end{rem}

\begin{cor}
\label{cor:vanishing}
Fix an integer $n > 1$.  Let $k$ be the function field of an $\aone$-connected variety $Y$.  Suppose $\underline{a}_1,\ldots,\underline{a}_r$ are a sequence of symbols of length $r$ in $k$.  Let $K$ be the function field of the product $k(Q_{\underline{a}_1} \times \cdots \times Q_{\underline{a}_r})$, and let $X$ be any smooth proper model of $K$ over $\cplx$.  For any integer $i < n$, $H^i_{ur}(X/\cplx,\mu_2^{\tensor i}) = 0$.
\end{cor}

\begin{rem}
There are (at least) two immediate problems with attempting to adapt the proof of Theorem \ref{thm:vanishing} to study lower degree unramified cohomology of Rost varieties.  First, as observed in Remark \ref{rem:retractrational}, we do not know if $Y_{\underline{a}}$ is retract rational, which makes description of the cokernel of the restriction map in unramified cohomology more difficult.  Second, the author is (perhaps due to his own ignorance) unaware of any simple result like Lemma \ref{lem:rewriting} that can be used to obtain ``normal forms" for Rost varieties over discrete valuation rings, though certainly Rost's chain lemma \cite{HaesemeyerWeibel} comes to mind.
\end{rem}

\section{A non-vanishing result}
In this section, for simplicity, we work over $\cplx$.  We combine the Orlov-Vishik-Voevodsky theorem together with Peyre's method to show that if $k$ is the function field of a well-chosen complex variety, and $K$ is a suitable product of small Pfister quadrics attached to symbols of length $n$, the non-trivial class in $H^n_{ur}(K/k,\mu_2^{\tensor n})$ guaranteed by Theorem \ref{thm:ovv} is actually contained in the subgroup $H^n_{ur}(K/\cplx,\mu_2^{\tensor n})$.

\subsubsection*{Peyre's results}
Suppose $F$ is the function field of a variety over $\cplx$.  Let $\ell$ be a prime number.  Following Peyre, we fix a finite dimensional ${\mathbb F}_{\ell}$-vector space $V$, and a morphism $\phi^1: V^{\vee} \to H^1_{\et}(F,\mu_{\ell})$.  The homomorphism $\phi^1$ extends to a morphism of graded ${\mathbb F}_{\ell}$-algebras
\[
\phi: \Lambda^{\cdot} (V^{\vee}) \longrightarrow H^i_{\et}(F,\mu_{\ell}^{\tensor i}).
\]
We fix an identification $\Lambda^i (V^{\vee}) \to (\Lambda^i V)^{\vee}$ by means of the map
\[
f_1 \wedge \cdots \wedge f_i \mapsto \{v_1 \wedge \cdots \wedge v_i \mapsto \sum_{\sigma \in S_i} sgn(\sigma) f_1(v_{\sigma(1)}) \cdots f_i(v_{\sigma(i)})\}.
\]
Having fixed this identification, for any basis $v_1,\ldots,v_n$ of $V$, with dual basis of $V^{\vee}$ given by $v_1^{\vee},\cdots,v_n^{\vee}$, the dual basis of $(v_{j_1} \wedge \cdots \wedge v_{j_i})_{j_1 < \cdots < j_i}$ is given by $(v_{j_1}^{\vee} \wedge \cdots \wedge v_{j_n}^{\vee})$.  With these choices, we get a morphism
\[
\phi^i: (\Lambda^i V)^{\vee} \longrightarrow H^i_{\et}(F,\mu_{\ell}^{\tensor i}).
\]

Let $S^i = \ker(\varphi^i)^{\perp} \subset \Lambda^i V$.  There is an induced morphism
\[
\hat{\varphi^i}: \hom(S^i,{\mathbb F}_{\ell}) \longrightarrow H^i_{\et}(F,\mu_{\ell}^{\tensor i}),
\]
which is necessarily injective.  We define $S^i_{dec} \subset S^i$ to be the subgroup of $S^i$ generated by elements of the form $v \wedge v'$ with $v \in V$ and $v' \in \Lambda^{i-1} V$.  With this notation, the next result is a special case of Peyre's results.

\begin{thm}[{\cite[Theorem 2 and Corollary 3]{Peyre1}}]
\label{thm:peyre}
If $f$ is an element of $\hom(S^i,{\mathbb F}_{\ell})$ such that $f|_{S^i_{dec}}$ is zero, then $\hat{\varphi}^i(f) \in H^i_{ur}(F/\cplx,\mu_{\ell}^{\tensor i})$.  Furthermore, if $S^i_{dec} \neq S^i$, then $H^i_{ur}(F/\cplx,\mu_{\ell}^{\tensor i}) \neq 0$.
\end{thm}

\subsubsection*{A construction}
Fix an integer $n > 0$. Let $k' = \cplx(t_1,\ldots,t_{2n})$ and $k = \cplx(t_1^{\ell},\ldots,t_{2n}^{\ell})$ and consider the inclusion $k \hookrightarrow k'$.  Let $V$ be an ${\mathbb F}_{\ell}$-vector space of dimension $2n$ with a chosen basis $v_1,\ldots,v_{2n}$; there is an isomorphism $V \isomt Gal(k'/k)$ and an injection $\phi^1: V^{\vee} \hookrightarrow H^1_{\et}(k,\mu_{\ell})$ that sends $v_j^{\vee}$ to the class of $t_j^{\ell}$.

If $K$ is a function field over $k$ that contains $k'$, we set $\phi^1_K$ to be the composite map
\[
\phi^1_K: V^{\vee} \longrightarrow H^1_{\et}(k,\mu_{\ell}) \longrightarrow H^1_{\et}(K,\mu_{\ell}).
\]
and set
\[
\phi^i_K: (\Lambda^i V)^{\vee} \longrightarrow H^i_{\et}(K,\mu_{\ell}^{\tensor i}).
\]
to be the map induced by exterior product on the left and cup product on the right.

As Peyre explains, the problem of finding a non-rational field with non-rationality detected by an unramified $\ell$-torsion class in the above setup translates in the following problem:  find a subspace $S \subset \Lambda^i V$ and an extension $K/k$ of function fields satisfying
\begin{itemize}
\item[i)] the kernel of the map $\phi^i_K: (\Lambda^i V)^{\vee} \to H^i_{\et}(K,\mu_{\ell}^{\tensor i})$ is $S^{\perp}$, and
\item[ii)] $S \neq S_{dec}$.
\end{itemize}

Define a degree $n$ unramified class as follows.  In terms of the basis $v_1,\ldots,v_{2n}$ described above, take a subspace $S$ of $\Lambda^n V$ spanned by the vector $v_1 \wedge \cdots \wedge v_{n} + v_{n+1} \wedge \cdots v_{2n}$.  The vector $v_1 \wedge \cdots \wedge v_{n} + v_{n+1} \wedge \cdots v_{2n}$ is not contained in $S_{dec}$.  Furthermore, we can explicitly compute
\[
S^{\perp} = \operatorname{Span}\left\langle\begin{matrix} v_{j_1}^{\vee} \wedge \cdots \wedge v_{j_n}^{\vee} \text{ for } 1 \leq j_1 < \cdots < j_n \leq 2n, \text{ with } \\ (j_1,\ldots,j_n) \notin \{(1,\ldots,n),(n+1,\ldots,2n)\}, \\ \text{ and } v_1^{\vee} \wedge \cdots \wedge v_n^{\vee} - v_{n+1}^{\vee} \wedge \cdots \wedge v_{2n}^{\vee} \end{matrix} \right\rangle.
\]
Thus, there exists a basis for $S^{\perp}$ where each element is of the form $w_{j_1} \wedge \cdots \wedge w_{j_n}$ for $w_{j_1}, \ldots, w_{j_n} \in V^{\vee}$.  Fix such a basis, and call it $I$.  Each element $s$ of $I$ defines a symbol
\[
\underline{s} := \hat{\varphi}^n(s) \in H^n_{\et}(k,\mu_{\ell}^{\tensor n}).
\]
We now use these symbols to provide the examples mentioned in the introduction.  First, we deal with the case $\ell = 2$.  The next result gives a more precise version of Theorem \ref{thmintro:mainexample} from the Introduction.

\begin{thm}
\label{thm:nonvanishing}
If $K = k(\prod_{s \in I} Q_{\underline{s}})$ (see \textup{Definition \ref{defn:pfisterforms}}), $X$ is any smooth proper model of $K$ over $\cplx$, and $m \geq 2$ is any integer, then $H^i_{ur}(X/\cplx,\mu_m^{\tensor i}) = 0$ for $i < n$, $H^n_{ur}(X/\cplx,\mu_2^{\tensor n}) \neq 0$, and $X$ is unirational.  In particular, $X$ is not $\aone$-connected.
\end{thm}

\begin{proof}
The vanishing of $H^i_{ur}(X/\cplx,\mu_2^{\tensor i})$ for $i < n$ is an immediate consequence of Corollary \ref{cor:vanishing}.

Since $K = k(\prod_{s \in I} Q_{\underline{s}})$, induction on the number of symbols together with Theorem \ref{thm:ovv} shows that the kernel of the restriction map
\[
H^n_{\et}(k,\mu_2^{\tensor n}) \longrightarrow H^n_{ur}(K/k,\mu_2^{\tensor n})
\]
is ${\Z/2}^{\times I}$ generated by the classes of the symbols $\underline{s}$.  The construction of $I$ given above then implies that the kernel of $\phi^n_K: (\Lambda^n V)^{\vee} \to H^n_{\et}(K,\mu_{\2}^{\tensor n})$ is $S^{\perp}$, and the construction of $S$ shows $S \neq S^{dec}$, so the non-triviality of $H^n_{ur}(X/\cplx,\mu_2^{\tensor n})$ is a consequence of Theorem \ref{thm:peyre}.

By definition, the image of any symbol $\underline{s} \in I$ in $H^n_{\et}(k,\mu_2^{\tensor n})$ in $H^n_{\et}(k',\mu_2^{\tensor n})$ is zero (see just after Theorem \ref{thm:peyre} for this notation).  Thus, the $Q_{\underline{s}}$ have a $k'$-rational point and are therefore $k'$-rational.  It follows that there is a degree $2$ extension of $K$ that is purely transcendental over $\cplx$.  That $X$ is not $\aone$-connected is then a consequence of \cite[Proposition 3.5 and Lemma 4.7]{ABirational}.
\end{proof}

\begin{rem}
A product of smooth quadrics splits after an extension of even degree.  Transfer techniques can then be used to show that the unramified cohomology with $\mu_m$-coefficients of such a product is trivial if $m$ is odd.  For $i < n$, the vanishing of unramified cohomology with $\mu_2$ coefficients can be used inductively to show that unramified cohomology of $\mu_{2^r}$ coefficients also vanishes.
\end{rem}

The following result provides a precise version of Theorem \ref{thmintro:lgeneric} from the introduction.

\begin{thm}
\label{thm:rcrostvarieties}
Suppose $\ell \neq 2$ is a prime.  Let $K = k(\prod_{s \in I} Y_{\underline{s}})$, where $Y_{\underline{s}}$ is the Rost variety attached to the symbol $\underline{s}$ (see \textup{Theorem \ref{thm:rostvarietiesexist}}).  If $X$ is a smooth projective model of $K/\cplx$, then $X$ is rationally connected, and $H^n_{ur}(X/\cplx,\mu_{\ell}^{\tensor n}) \neq 0$.  In particular, $X$ is not $\aone$-connected.
\end{thm}

\begin{proof}
In this case, the rational connectivity of $X$ follows from the statement of Proposition \ref{cor:rostvarietiesarerationallyconnected}.  Since $K = k(\prod_{s \in I} Y_{\underline{s}})$, an induction argument using Theorem \ref{thm:merkurjevsuslin} shows that the kernel of the restriction map
\[
H^n_{\et}(k,\mu_{\ell}^{\tensor n}) \longrightarrow H^n_{ur}(K/k,\mu_{\ell}^{\tensor n})
\]
is $\Z/\ell^{I}$ generated by the classes of the symbols $\underline{s}$.  Again, the construction of $I$ given above then implies that the kernel of $\phi^n_K: (\Lambda^n V)^{\vee} \to H^n_{\et}(K,\mu_{\ell}^{\tensor n})$ is $S^{\perp}$, and the construction of $S$ shows $S \neq S^{dec}$, so the non-triviality of the unramified cohomology group in question is an immediate consequence of \ref{thm:peyre}.  That $X$ is not $\aone$-connected is then a consequence of \cite[Proposition 3.5 and Lemma 4.7]{ABirational}.
\end{proof}

\subsubsection*{Unramified cohomology of rationally connected varieties}
In the above examples, if we fix a prime number $\ell$ and an integer $n$, the rationally connected variety $X$ with an unramified degree $n$ cohomology class modulo $\ell$ has dimension at least $\ell^n-1 + 2n$.  In particular, this number grows rapidly with $n$ for a fixed $\ell$.

\begin{question}
If $X$ is a rationally connected smooth proper complex variety of fixed dimension $d$, are there restrictions on the integers $n$ and primes $\ell$ for which $H^n_{ur}(X/\cplx,\mu_{\ell}^{\tensor n})$ can be non-zero?
\end{question}

Two restrictions on the possible values of $n$ are presented in the following remarks.

\begin{rem}
The unramified cohomology of an irreducible smooth proper complex variety of dimension $d$ vanishes in degrees greater than $d$.  Indeed, if $k$ is a field, write $G_k$ for the absolute Galois group of $k$, $cd_p(G_k)$ for the $p$-cohomological dimension of $G_k$ and $cd(G_k)$ for $sup_p cd_p(G_k)$ (see \cite[I.3]{Serre}).  Combining \cite[II 4.1 Proposition 10' and II 4.2 Proposition 11]{Serre}, one observes that if $X$ is an irreducible smooth complex algebraic variety of dimension $d$, then $cd(G_{\cplx(X)}) = d$.  Since $H^i_{ur}(X/\cplx,\mu_{\ell}^{\tensor n})$ is by definition a subgroup of $H^i_{\et}(\cplx(X),\mu_j^{\tensor n})$ the stated vanishing follows.
\end{rem}

\begin{rem}
For arbitrary $\ell$, if $X$ is a rationally connected smooth proper complex variety, we know that $H^n_{ur}(X/\cplx,\mu_{\ell}^{\tensor n})$ is non-zero only if $n > 1$.  Indeed, if $X$ is a rationally connected smooth proper complex variety, then $H^1_{ur}(X/\cplx,\mu_n) = H^1_{\et}(X,\mu_n) = 0$ for arbitrary $n$ since such varieties are (topologically or \'etale) simply connected.
\end{rem}

Since all rationally connected complex curves and surfaces are rational, the first interesting case of the question is the case where $X$ is a smooth proper rationally connected $3$-fold.  In that case, the question can be given a much more precise form. 

\begin{rem}
If $X$ is a rationally connected smooth proper complex $3$-fold, then a result of Colliot-Th\'el\`ene-Voisin \cite[Corollaire 6.2]{ColliotTheleneVoisin}, which uses Voisin's affirmation of the integral Hodge conjecture for rationally connected $3$-folds, implies that $H^3_{ur}(X/\cplx,\mu_\ell^{\tensor 3})$ is trivial for arbitrary $\ell$.  Thus, if $X$ is a smooth proper rationally connected complex $3$-fold, then only the second degree unramified cohomology can be non-trivial.
\end{rem}

\begin{question}
If $X$ is a smooth proper rationally connected complex $3$-fold, for which primes $\ell$ can $H^2_{ur}(X/\cplx,\mu_{\ell})$ be non-trivial?
\end{question}

\begin{rem}
A natural way to attack the previous problem is to use the minimal model program.  The output of the minimal model program is a Fano fibration, which under the rational connectivity assumption must have rationally connected base.  If the fiber dimension is $1$, since rationally connected surfaces are rational, the varieties in question are conic bundles over rational surfaces.  In this case, the Artin-Mumford example shows that $2$-torsion can appear in the unramified Brauer group.  Moreover, in such examples, only $2$-torsion can appear by \cite[Proposition 3]{ArtinMumford} (see also \cite[Theorem 8.3.3]{IskovkikhProkhorov}).  If the fiber dimension is $2$, the varieties in question are models of geometrically rational surfaces $X$ over $\cplx(t)$.  The Brauer group of a model over $\cplx$ of such a variety injects into the Brauer group of $X$ over $\cplx(t)$.  In the latter case, results of Manin \cite{Manin} can be used to show that only $2$, $3$ and $5$ torsion can appear.  Using a cubic surface over $\cplx(t)$, is it possible to produce a rationally connected $3$-fold with $3$-torsion in the Brauer group?  Likewise, using a del Pezzo surface of degree $1$ over $\cplx(t)$, is it possible to produce a rationally connected $3$-fold with $5$-torsion in the Brauer group?   Brauer groups of non-singular Fano $3$-folds are apparently trivial \cite[p. 168]{IskovkikhProkhorov}, but for rationally connected smooth proper varieties whose minimal models are singular Fano $3$-folds, the Brauer group can be non-trivial.
\end{rem}

\begin{footnotesize}
\bibliographystyle{alpha}
\bibliography{nonrationality}
\end{footnotesize}
\end{document}